\documentclass[noamsfonts,a4paper]{amsart}
\usepackage[margin=2cm]{geometry}
\usepackage[bitstream-charter,cal,sfscaled=false]{mathdesign}
\usepackage{hyperref}
\theoremstyle{plain}
\newtheorem{theorem}{Theorem}[section]
\newtheorem{lemma}[theorem]{Lemma}
\newtheorem{corollary}[theorem]{Corollary}
\theoremstyle{definition}

\title{A solution to Haagerup's problem: Normality of weights is commutatively determined}

\author[D. Blecher]{David P. Blecher}
\address{Department of Mathematics\\University of Houston\\Houston, TX 77204-3008, USA}
\email{dpbleche@central.uh.edu}

\author[I. Choi]{Ikhan Choi}
\address{Graduate School of Mathematical Sciences\\The University of Tokyo\\3-8-1 Komaba, Meguro-ku, Tokyo, 153-8914, Japan}
\email{choi@ms.u-tokyo.ac.jp}

\subjclass[2020]{46L51, 46L05, 46L10}
\keywords{von Neumann algebras, normal weights, C$^*$-algebras, lower semi-continuous weights}

\begin{document}

\begin{abstract}
We prove that a weight on a C$^*$-algebra is lower semi-continuous if it is lower semi-continuous on commutative C$^*$-subalgebras.
This can be used to prove a von Neumann algebraic counterpart that a weight on a von Neumann algebra is normal if it is normal on commutative von Neumann subalgebras.
This provides an affirmative solution to Haagerup's second problem (Problem 1.11) on normal weights posed in 1975.
\end{abstract}

\maketitle

\section{Introduction}

Weights on operator algebras may be viewed as a natural setting for ``quantum integrals'', usually with an appropriate lower semi-continuity expressing a non-commutative form of the classical Fatou lemma in integration theory.
For example, $\sigma$-weakly lower semi-continuous weights on von Neumann algebras and norm lower semi-continuous weights on C$^*$-algebras generalize the localizable measures and Radon measures respectively.
In this paper we will establish the following theorems:

\begin{theorem}\label{Theorem: C*}
Let $\varphi$ be a weight on a C$^*$-algebra.
Then, the following are all equivalent.
\begin{enumerate}
\item $\varphi$ preserves norm convergent positive sums on commutative C$^*$-subalgebras.
\item $\varphi$ preserves norm convergent positive sums.
\item $\varphi$ is (norm) lower semi-continuous.
\item $\varphi$ is a pointwise supremum of positive linear functionals.
\end{enumerate}
\end{theorem}

\begin{theorem}\label{Theorem: W*}
Let $\varphi$ be a weight on a von Neumann algebra.
Then, the following are all equivalent.
\begin{enumerate}
\item $\varphi$ preserves $\sigma$-weakly convergent positive sums on commutative von Neumann subalgebras.
\item $\varphi$ preserves $\sigma$-weakly convergent positive sums.
\item $\varphi$ is $\sigma$-weakly lower semi-continuous.
\item $\varphi$ is a pointwise supremum of normal positive linear functionals.
\item[(2.5)] $\varphi$ is normal, in the sense that it preserves directed suprema.
\end{enumerate}
\end{theorem}

We recall that $\varphi:M_+\to[0,\infty]$ is a weight if it is additive, positively homogeneous (hence monotone), with $\varphi(0)=0$.
Note that the sums in (2) are over arbitrary families, not just countable ones.
The `decreasing directions' from (4) to (1) in both theorems are all obvious or immediate by the obvious calculation.
For example that (3) implies (2) in both theorems follows from the liminf formulation of lower semi-continuity applied to the net of partial sums of the series.
The equivalence between (3) and (4) in Theorem \ref{Theorem: C*} is usually referred to as the {\em  Combes theorem,} originally proved in \cite{MR236721}.
For Theorem \ref{Theorem: W*}, the equivalence of (2), (3), (4), and (2.5), is part of Haagerup's main result in \cite{MR380438}.
Consequently, the main contributions of our paper to these results are the equivalence of (1) and (3) in Theorem \ref{Theorem: C*}, and the equivalence of (1) and (2) in Theorem \ref{Theorem: W*}.
We will use the Combes theorem (from (3) to (4)) and  Haagerup's  equivalence (from (4) to (3)) (see \cite{MR4864224} and \cite{MR4972190} for modern proofs of these) in our Theorems \ref{Theorem: C*} and \ref{Theorem: W*}.

In his first paper \cite{MR380438}, Haagerup posed three problems on normal weights: Problem 1.10, 1.11, and 2.7.
The second problem, numbered  1.11, asks whether normality of weights is commutatively determined.
That is, does normality follow from normality on abelian von Neumann subalgebras?
This is immediate from the equivalence of  (1) and (2.5) in  Theorem \ref{Theorem: W*}.
(The variant of this for `sequential normality' is also true, and will appear in \cite{blecher2026quantumintegralswithappendix}.)
The other two problems have been recently resolved in \cite{MR4972190} and \cite{choi2026counterexamplehaagerupsproblemsubadditive}.
As a result, we may now close the book on Haagerup's three fifty-year-old problems on normal weights.

We will call the equivalence with (1) in the theorems above by the name {\em  commutative detection}.
One of the most striking results we discovered  is that there is a natural proof for commutative detection for von Neumann algebras which uses norm lower semi-continuity of weights.  This means that Haagerup's second problem was mainly essentially C$^*$-algebraic, rather than von Neumann algebraic.
In the last part of Section 2, we will prove that a weight on a von Neumann algebra which is (norm) lower semi-continuous and completely additive on projections satisfies (4) of Theorem \ref{Theorem: W*}.

\subsection*{Preliminaries}

A weight $\varphi$ on a von Neumann algebra $M$ is  \emph{normal} if it preserves directed suprema in the sense that for a bounded increasing net $x_i\in M_+$ we have $\varphi(\sup_ix_i)=\sup_i\varphi(x_i)$.
This property is also equivalent to the properties in Theorem \ref{Theorem: W*}.  Indeed  a $\sigma$-weakly lower semi-continuous weight is clearly normal. Conversely, a normal weight  preserves $\sigma$-weakly convergent positive sums.
This definition of normality does not have an analogue for completely general C$^*$-algebras, but it is quite meaningful when one compares  normality  and sequential normality of weights on von Neumann algebras.
See \cite{blecher2026quantumintegralswithappendix}.

A \emph{semi-cyclic representation} of a C$^*$-algebra $A$ is a representation $\pi:A\to B(H)$ on a Hilbert space $H$ together with a partially defined left $A$-linear map $\Lambda:\mathrm{dom}\,\Lambda\subset A\to H$ with dense range.
For a weight on a C$^*$-algebra, there is a canonically associated semi-cyclic representation also called Gelfand-Naimark-Segal representation (pp.\ 77-79 in \cite{MR236721}).
Indeed the unitary equivalence classes of semi-cyclic representations on von Neumann algebras are in bijective correspondence with the weights (the converse follows by adapting the argument in the proof of  Lemma 1.1 in \cite{MR380438}).
We always let $(\pi_\varphi,\Lambda_\varphi)$ be the semi-cyclic representation of $A$ associated with $\varphi$ for a weight $\varphi$ on a C$^*$-algebra $A$.
We also use the standard notations $\mathfrak{N}_\varphi:=\{a\in A:\varphi(a^*a)<\infty\}$ and $\mathfrak{M}_\varphi:=\mathrm{span}\,\{b^*a\in A:a,b\in\mathfrak{N}_\varphi\}$ without mention.
We will use the facts without mentions that $\mathfrak{N}_\varphi$ is a left ideal and $\mathfrak{M}_\varphi$ is a hereditary subalgebra of $A$.
It is well-known that if $\Lambda_\varphi$ is closed in the sense that the graph is norm closed in $A\times H$, then $\varphi$ is lower semi-continuous (see the argument for Lemma 6.8 (iii) in \cite{blecher2026quantumintegralswithappendix}).
Similarly, a simple variant of this argument shows that for a weight $\varphi$ on a von Neumann algebra $M$, $\varphi$ is normal if the graph $\Lambda_\varphi$ is closed in $M\times H$, where $M$ is endowed with the $\sigma$-weak topology.  One uses bounded increasing nets instead of sequences, and the fact that the graph being closed in $M\times H$ where $H$ has the norm topology,  is equivalent to the same fact where $H$ has the weak topology.  The heart of any proof of `commutative detection' for a weight $\phi$ seems to be establishing the closedness of $\Lambda_\varphi$.

The positive cone of an ordered vector space $E$ will be denoted by $E_+$.

\subsection*{AI disclosure statement}
The proofs were substantially obtained with the aid of GPT-6.0, and checked and rewritten by the authors.  
The authors take full responsibility for the content and correctness of this paper.

\subsection*{Acknowledgement} 
The first author was supported by NSF Grant DMS-2154903, and thanks Louis Labuschagne for discussions and allowing the temporary placement of some of the results here in an Appendix in \cite{blecher2026quantumintegralswithappendix}. 
The second author was supported by JSPS KAKENHI Grant Number JP26KJ0751.

\section{Commutative detection}

\begin{lemma}[Right orthogonal sequences]\label{Lemma: Right orthogonal sequences}
Let $\varphi$ be a weight on a C$^*$-algebra $A$ which preserves norm convergent positive sums on commutative C$^*$-subalgebras.
Let $a_n\in\mathfrak{N}_\varphi$ be a right orthogonal sequence in the sense that $a_ma_n^*=0$ for every $m\ne n$, and such that $\sum_n(\|a_n\|+\|\Lambda_\varphi(a_n)\|)<\infty$.
Then $\sum_na_n\in\mathfrak{N}_\varphi$ and $\Lambda_\varphi(\sum_na_n)=\sum_n\Lambda_\varphi(a_n)$.
\end{lemma}
\begin{proof}
Assume $a_n\ne0$ for all $n$ and let $r_n:=\|a_n\|+\|\Lambda_\varphi(a_n)\|>0$.
Since the sum $\sum_nr_n^{-1}a_n^*a_n$ is norm convergent, and since the sequence $a_n^*a_n$ generate a commutative subalgebra of $A$, by the assumption on $\varphi$ we have
\[\varphi(\sum_nr_n^{-1}a_n^*a_n)=\sum_nr_n^{-1}\varphi(a_n^*a_n)=\sum_nr_n^{-1}\|\Lambda_\varphi(a_n)\|^2\le\sum_n\|\Lambda_\varphi(a_n)\|<\infty.\]
By expansion of products we have a simple  operator Cauchy-Schwarz inequality (see e.g.\ the proof of \cite[Lemma 6.8]{blecher2026quantumintegralswithappendix}) 
\[|\sum_na_n|^2\le(\sum_nr_n) \, (\sum_nr_n^{-1}a_n^*a_n).\]
That $\sum_na_n\in\mathfrak{N}_\varphi$ follows from the last inequality: 
\[\varphi(|\sum_na_n|^2)\le(\sum_nr_n) \, \varphi(\sum_nr_n^{-1}a_n^*a_n)\le(\sum_nr_n)\sum_n\|\Lambda_\varphi(a_n)\|<\infty.\]
Similarly, that $\Lambda_\varphi(\sum_na_n)=\sum_n\Lambda_\varphi(a_n)$ follows in the same way from  
\[\|\Lambda_\varphi(\sum_{n>m}a_n)\|^2=\varphi(|\sum_{n>m}a_n|^2)\le(\sum_{n>m}r_n) \, \varphi(\sum_{n>m}r_n^{-1}a_n^*a_n)\le(\sum_{n>m}r_n)\sum_{n>m}\|\Lambda_\varphi(a_n)\|\to0,\qquad m\to\infty.\qedhere\]
\end{proof}

\begin{theorem}[(1)$\Rightarrow$(3) in Theorem \ref{Theorem: C*}]
Let $\varphi$ be a weight on a C$^*$-algebra.
If $\varphi$ preserves norm convergent positive sums on commutative C$^*$-subalgebras, then it is lower semi-continuous.
\end{theorem}

\begin{proof}
For $\delta\in(0,1)$, define piecewise linear continuous functions $f_\delta,g_\delta,h_\delta:[0,\infty)\to[0,1]$ by
\[f_\delta(t):=\left\{\begin{alignedat}{2}&\delta^{-1}t&\quad&\text{ if }t\le\delta,\\&1&&\text{ if }t\ge\delta,\end{alignedat}\right.\quad
g_\delta(t):=\left\{\begin{alignedat}{2}&0&\quad&\text{ if }t\le\delta,\\&\delta^{-1}(t-\delta)&&\text{ if }\delta\le t\le 2\delta,\\&1&&\text{ if }t\ge2\delta,\end{alignedat}\right.\quad
h_\delta(t):=\left\{\begin{alignedat}{2}&0&\quad&\text{ if }t\le\delta,\\&(1-\delta)^{-1}(t-\delta)&&\text{ if }\delta\le t\le1,\\&1&&\text{ if }t\ge1,\end{alignedat}\right.\]
Let $\varphi$ be a weight on a C$^*$-algebra $A$ which preserves norm convergent positive sums on commutative C$^*$-subalgebras.
It suffices  (see the end of the introduction) to prove that $\Lambda_\varphi$ is norm closed.

(Closability)
Suppose there is a sequence $a_m\in\mathfrak{N}_\varphi$ such that $\|a_m\|+\|\Lambda_\varphi(a_m)-\xi\|\to0$ as $m\to\infty$ for some $\xi\in H_\varphi$.
It suffices to show $\xi=0$.
We construct a right orthogonal sequence $c_n\in\mathfrak{N}_\varphi$ satisfying a summability condition
\[\sum_{n=1}^\infty(\|c_n\|+\|\Lambda_\varphi(c_n)-\xi\|)<\infty.\]
Assume that such a sequence $c_n$ exists, and let $c_0:=0$.
Then, the sequences $(c_{2n}-c_{2n-1})_n$ and $(c_{2n-1}-c_{2n-2})_n$ are right orthogonal, and
\[\sum_{n=1}^\infty(\|c_{2n}-c_{2n-1}\|+\|\Lambda_\varphi(c_{2n}-c_{2n-1})\|)<\infty,\qquad\sum_{n=1}^\infty(\|c_{2n-1}-c_{2n-2}\|+\|\Lambda_\varphi(c_{2n-1}-c_{2n-2})\|)<\infty.\]
Applying the lemma on right orthogonal sequences (Lemma \ref{Lemma: Right orthogonal sequences}) with the assumption on $\varphi$, it follows that
\[\xi=\lim_{n\to\infty}\Lambda_\varphi(c_{2n})=\Lambda_\varphi(\sum_{n=1}^\infty(c_{2n}-c_{2n-1}))+\Lambda_\varphi(\sum_{n=1}^\infty(c_{2n-1}-c_{2n-2}))=\Lambda_\varphi(0)=0,\]
so the proof is complete.
To obtain the right orthogonality of $c_n$, we simultaneously construct an auxiliary sequence $e_n\in\mathfrak{N}_\varphi$ such that $0\le e_n\le1$ and $c_ie_n=c_i$ for all $i\le n$.

Fix an induction stage $n$ to assume that we have constructed up to $c_n$ and $e_n$.
Let $\varepsilon>0$ be a constant depending on $n$ determined later.
Choose $a_{n+1}\in\mathfrak{N}_\varphi$ in the sequence $a_m$ such that
\[\|a_{n+1}\|<\varepsilon,\qquad\|\Lambda_\varphi(a_{n+1})-\xi\|<\varepsilon.\]
Define
\[b_{n+1}:=a_{n+1}(1-f_{1/2}(e_n))\in\mathfrak{N}_\varphi.\]
Then, $g_{1/2}(e_n)$ annihilates it by right multiplication, and
\begin{align*}
\|b_{n+1}-a_{n+1}\|&=\|a_{n+1}f_{1/2}(e_n)\|<\varepsilon,\\
\|\Lambda_\varphi(b_{n+1})-\Lambda_\varphi(a_{n+1})\|&=\|\Lambda_\varphi(a_{n+1}f_{1/2}(e_n))\|<\varepsilon\|\Lambda_\varphi(f_{1/2}(e_n))\|.
\end{align*}
Define
\[c_{n+1}:=b_{n+1}g_\delta(|b_{n+1}|^2)\in\mathfrak{N}_\varphi,\]
so that $g_{1/2}(e_n)$ still annihilates it but now $f_\delta(|b_{n+1}|^2)$ acts on it as a unit by right multiplication, where $\delta\in(0,1)$ is sufficiently small such that
\begin{align*}
\|c_{n+1}-b_{n+1}\|&=\|b_{n+1}(1-g_\delta(|b_{n+1}|^2))\|=\||b_{n+1}|^2(1-g_\delta(|b_{n+1}|^2))^2\|^{\frac12}<\varepsilon,\\
\|\Lambda_\varphi(c_{n+1})-\Lambda_\varphi(b_{n+1})\|&=(\varphi(|b_{n+1}|^2)-\varphi(|b_{n+1}|^2(1-(1-g_\delta(|b_{n+1}|^2))^2)))^{\frac12}<\varepsilon.
\end{align*}
The latter is possible by the assumption on $\varphi$, because the non-decreasing sequence of continuous functions $t\mapsto t(1-(1-g_{m^{-1}}(t))^2)$ uniformly converges to the the identity $t\mapsto t$ on $t\in[0,\|b_{n+1}\|^2]$ as $m\to\infty$ so that the convergence can be represented as a norm convergent sum by the telescoping.
Then, combining all the estimates above, we get by choosing $\varepsilon:=(5+\|\Lambda_\varphi(f_{1/2}(e_n))\|)^{-1}2^{-(n+1)}$ the summability
\[\|c_{n+1}\|+\|\Lambda_\varphi(c_{n+1})-\xi\|<2^{-(n+1)}.\]

Define
\[e_{n+1}:=f_\delta(|b_{n+1}|^2)+g_{1/2}(e_n)\in\mathfrak{N}_\varphi.\]
Then, the conditions $0\le e_{n+1}\le1$ and $c_ie_{n+1}=c_i$ for $i\le n+1$ can be checked as follows.
\begin{enumerate}
\item Since $|b_{n+1}|^2g_{1/2}(e_n)=b_{n+1}^*a_{n+1}(1-f_{1/2}(e_n))g_{1/2}(e_n)=0$ implies $f_\delta(|b_{n+1}|^2)g_{1/2}(e_n)=0$, we have $e_{n+1}\le1$.
\item Since $c_ie_n=c_i$ for $i\le n$ by the induction hypothesis, the functional calculus with $g_{1/2}(1)=1$ and the orthogonality (1) give
\[c_ig_{1/2}(e_n)=c_i,\qquad c_if_\delta(|b_{n+1}|^2)=0\]
respectively, so $c_ie_{n+1}=c_i$ for $i\le n$.
\item 
Since
\[c_{n+1}f_\delta(|b_{n+1}|^2)=c_{n+1},\qquad c_{n+1}g_{1/2}(e_n)=0\]
by definition of $c_{n+1}$, we have $c_{n+1}e_{n+1}=c_{n+1}$.
\end{enumerate}
The right orthogonality of $c_n$ follows as $c_ic_{n+1}^*=c_ig_{1/2}(e_n)c_{n+1}^*=0$ for $i\le n$.
This completes the construction of sequences $c_n$ and $e_n$.

(Closedness)
Suppose there is a sequence $a_m\in\mathfrak{N}_\varphi$ such that $\|a_m-x\|+\|\Lambda_\varphi(a_m)-\xi\|\to0$ as $m\to\infty$ for some $x\in A$ and $\xi\in H_\varphi$.
We need to prove $x\in\mathfrak{N}_\varphi$ and $\xi=\Lambda_\varphi(x)$.
We construct a sequence $h_n\in\mathfrak{M}_\varphi$ such that $h_nh_{n+1}=h_n$ for all $n$ and we have the summability
\[\sum_{n=1}^\infty(\|x-xh_n\|+\|\xi-\Lambda_\varphi(xh_n)\|)<\infty.\]
Assume that such a sequence $h_n$ exists, and let $c_n:=xh_n$ with $h_0=0$.
Then, we can easily check the sequences $(c_{2n}-c_{2n-1})_n$ and $(c_{2n-1}-c_{2n-2})_n$ are right orthogonal, and
\[\sum_{n=1}^\infty(\|c_{2n}-c_{2n-1}\|+\|\Lambda_\varphi(c_{2n}-c_{2n-1})\|)<\infty,\qquad\sum_{n=1}^\infty(\|c_{2n-1}-c_{2n-2}\|+\|\Lambda_\varphi(c_{2n-1}-c_{2n-2})\|)<\infty.\]
Applying the lemma on right orthogonal sequences (Lemma \ref{Lemma: Right orthogonal sequences}) with the assumption on $\varphi$, we have
\[x=\lim_{n\to\infty}c_{2n}=\sum_{n=1}^\infty(c_{2n}-c_{2n-1})+\sum_{n=1}^\infty(c_{2n-1}-c_{2n-2})\in\mathfrak{N}_\varphi\]
and
\[\xi=\lim_{n\to\infty}\Lambda_\varphi(c_{2n})=\Lambda_\varphi(\sum_{n=1}^\infty(c_{2n}-c_{2n-1}))+\Lambda_\varphi(\sum_{n=1}^\infty(c_{2n-1}-c_{2n-2}))=\Lambda_\varphi(x),\]
so the proof is complete.
To obtain the unit sequence condition $h_nh_{n+1}=h_n$, we simultaneously construct an auxiliary sequence $e_n\in\mathfrak{M}_\varphi$ such that $h_ne_n=h_n$ for all $n$.

Suppose we have constructed up to $h_n$ and $e_n$.
Define
\[y:=x-xe_n\in A,\qquad\eta:=\xi-\Lambda_\varphi(xe_n)\in H_\varphi,\qquad b_m:=a_m-a_me_n\in\mathfrak{N}_\varphi.\]
Then, $\|b_m-y\|+\|\Lambda_\varphi(b_m)-\eta\|\to0$ as $m\to\infty$ by the assumption of $x$ and $\xi$.
If we let $\lambda:=\sup_\delta\pi(f_\delta(yy^*))\in B(H_\varphi)$ be the left support projection of $\pi(yy^*)$ taken by the strong limit in the von Neumann algebra $B(H_\varphi)$, then since in the limit $(\delta,m)\to(0,\infty)$ we have
\begin{align*}
\|(1-f_\delta(yy^*))(a_m-xe_n)\|
&=\|(1-f_\delta(yy^*))(a_m-x)+(1-f_\delta(yy^*))y\|\\
&\le\|a_m-x\|+\|(1-f_\delta(yy^*))^2yy^*\|^{\frac12}\to0,
\end{align*}
and similarly
\begin{align*}
\|\Lambda_\varphi((1-f_\delta(yy^*))(a_m-xe_n))-(\eta-\lambda\eta)\|
&=\|\pi_\varphi(1-f_\delta(yy^*))(\Lambda_\varphi(a_m)-\xi)+\pi_\varphi(1-f_\delta(yy^*))\eta-(\eta-\lambda\eta)\|\\
&\le\|\Lambda_\varphi(a_m)-\xi\|+\|(\lambda-\pi(f_\delta(yy^*)))\eta\|\to0,
\end{align*}
so the closability of $\Lambda_\varphi$ gives $\eta=\lambda\eta$.
This gives a limit
\[\|\eta-\pi(f_\delta(yy^*))\eta\|=\|(\lambda-\pi(f_\delta(yy^*)))\eta\|\to0,\qquad\delta\to0.\]
Let $\varepsilon>0$ be a constant determined later, and fix $\delta>0$ small enough such that
\[\|y-f_\delta(yy^*)y\|<\varepsilon,\qquad\|\eta-\pi(f_\delta(yy^*))\eta\|<\varepsilon.\]
Denoting by $\grave f_\delta(t):=\max\{t,\delta\}^{-1}$ so that $t\grave f_\delta(t)=f_\delta(t)$ for all $t\ge0$, as $m\to\infty$ we get the limit superior
\begin{align*}
\|y-yf_\delta(b_m^*b_m)\|
&=\|(y-f_\delta(yy^*)y)+y(f_\delta(y^*y)-f_\delta(b_m^*b_m))\|\\
&<\varepsilon+\|y\|\|f_\delta(y^*y)-f_\delta(b_m^*b_m)\|\to\varepsilon,
\end{align*}
and similarly but with the boundedness of $\|\Lambda_\varphi(b_m)\|$ we get
\begin{align*}
\|\eta-\Lambda_\varphi(yf_\delta(b_m^*b_m))\|
&=\|(\eta-\pi(f_\delta(yy^*))\eta)+\pi(f_\delta(yy^*))(\eta-\Lambda_\varphi(b_m))+\pi(f_\delta(yy^*))\Lambda_\varphi(b_m)-\Lambda_\varphi(yf_\delta(b_m^*b_m))\|\\
&\le\varepsilon+\|\eta-\Lambda_\varphi(b_m)\|+\|y\grave f_\delta(y^*y)y^*-y\grave f_\delta(b_m^*b_m)b_m^*\|\sup_m\|\Lambda_\varphi(b_m)\|\to\varepsilon.
\end{align*}
It implies that we can find
\[f_{n+1}:=f_\delta(b_m^*b_m)\le\delta^{-1}b_m^*b_m\in\mathfrak{M}_\varphi\]
by taking large $m$ such that
\[\|y-yf_{n+1}\|<2\varepsilon,\qquad\|\eta-\Lambda_\varphi(yf_{n+1})\|<2\varepsilon.\]

Now we use $h_n$.
Define
\begin{align*}
g_{n+1}:&=1-(1-e_n)(1-f_{n+1})(1-e_n)\\
&=2e_n+f_{n+1}-e_n^2-e_nf_{n+1}-f_{n+1}e_n+e_nf_{n+1}e_n\\
&\le2e_n+f_{n+1}+f_{n+1}^2+e_nf_{n+1}e_n\in\mathfrak{M}_\varphi
\end{align*}
so that
\begin{align*}
\|x-xg_{n+1}\|&=\|y(1-f_{n+1})(1-e_n)\|\le\|y-yf_{n+1}\|<2\varepsilon,\\
\|\xi-\Lambda_\varphi(xg_{n+1})\|
&=\|\eta-\Lambda_\varphi(x(g_{n+1}-e_n))\|\\
&=\|\eta-\Lambda_\varphi(x(1-e_n)(f_{n+1}+(1-f_{n+1})e_n))\|\\
&\le\|\eta-\Lambda_\varphi(yf_{n+1})\|+\|y-yf_{n+1}\|\|\Lambda_\varphi(e_n)\|<2\varepsilon(1+\|\Lambda_\varphi(e_n)\|).
\end{align*}
With a new choice of sufficiently small $\delta>0$, define
\[h_{n+1}:=h_\delta(g_{n+1})\in\mathfrak{M}_\varphi\]
such that $0\le g_{n+1}-h_{n+1}\le\delta$ and $0\le(g_{n+1}-h_{n+1})^2\le\delta g_{n+1}$ give
\begin{align*}
\|xg_{n+1}-xh_{n+1}\|&\le\delta\|x\|<\varepsilon,\\
\|\Lambda_\varphi(xg_{n+1})-\Lambda_\varphi(xh_{n+1})\|&\le\delta^{\frac12}\|x\|\varphi(g_{n+1})^{\frac12}<\varepsilon.
\end{align*}
Then, combining all the estimates above, we get by choosing $\varepsilon:=(6+2\|\Lambda_\varphi(e_n)\|)^{-1}2^{-(n+1)}$ the summability
\[\|x-xh_{n+1}\|+\|\xi-\Lambda_\varphi(xh_{n+1})\|<2^{-(n+1)}.\]

Finally, if we define
\[e_{n+1}:=f_\delta(g_{n+1})\le\delta^{-1}g_{n+1}\in\mathfrak{M}_\varphi\]
with $\delta$ in the definition of $h_{n+1}$, then the identity $h_\delta f_\delta=h_\delta$ implies the condition $h_{n+1}e_{n+1}=h_\delta(g_{n+1})f_\delta(g_{n+1})=h_\delta(g_{n+1})=h_{n+1}$ for $e_{n+1}$.
Since the induction hypothesis $h_ne_n=h_n$ implies $h_ng_{n+1}=h_n$ because $1-e_n$ factors $1-g_{n+1}$, by the functional calculus with $h_\delta(1)=1$, the condition $h_nh_{n+1}=h_nh_\delta(g_{n+1})=h_n$ for $h_{n+1}$ is satisfied.
This completes the proof.
\end{proof}

\begin{corollary}[C$^*$-algebraic Haagerup's second problem] \label{cor:commutative-lsc} 
A weight on a C$^*$-algebra that is lower semi-continuous on commutative C$^*$-subalgebras is lower semi-continuous.
\end{corollary}

We give an additional characterization of lower semi-continuity of weights on C$^*$-algebras.

\begin{corollary}[Spectral recovery and lower semi-continuity]
\label{cor:spectral-recovery-lsc}
Let $\varphi$ be a weight on a C$^*$-algebra $A$.
Then $\varphi$ is norm lower semi-continuous if and only if
\[
\varphi(a)=\sup_{\varepsilon>0}\varphi((a-\varepsilon)_+), 
\qquad a\in A_+,
\]
where $(a-\varepsilon)_+$ is defined by continuous functional
calculus. If $A=M$ is a von Neumann algebra, these conditions
are also equivalent to the spectral finite-spectrum recovery
property $(FS_c)$:
\[
\varphi(a)=\sup\{\varphi(s):0\le s\le a,\ 
s\in W^*(1,a),\ s\text{ has finite spectrum}\},
\qquad a\in M_+.
\]
\end{corollary}

\begin{proof}
Norm lower semi-continuity and monotonicity imply the first displayed 
identity, since $(a-\varepsilon)_+\to a$ in norm as
$\varepsilon\downarrow0$. In the von Neumann algebra case,
the same argument applied to increasing finite-spectrum
spectral approximants converging to $a$ in norm gives $(FS_c)$.

Conversely, suppose the first displayed 
identity holds.
Let $C\subseteq A$ be a commutative C$^*$-subalgebra, and
let $a_n,a\in C_+$ with $a_n\to a$ in norm.
For every $\varepsilon>0$, eventually
$a_n\ge a-\varepsilon1$ in the unitization of $C$.
Since $a_n\ge0$ and these elements commute,
$a_n\ge(a-\varepsilon)_+$. Hence
\[
\liminf_n\varphi(a_n)\ge\varphi((a-\varepsilon)_+).
\]
Taking the supremum over $\varepsilon>0$ proves norm lower
semi-continuity on $C_+$. Corollary \ref{cor:commutative-lsc}
therefore gives norm lower semi-continuity on $A_+$.

Finally, suppose $A=M$ and $(FS_c)$ holds.
Fix $0\le s\le a$ with $s\in W^*(1,a)$ of finite spectrum.
Finite additivity and positive homogeneity give
$\varphi(s)=\sup_{\varepsilon>0}\varphi((s-\varepsilon)_+)$.  This 
includes the case when $\varphi(s)=\infty$.
Since $s$ and $a$ commute,
$(s-\varepsilon)_+\le(a-\varepsilon)_+$, so
\[
\varphi(s)\le
\sup_{\varepsilon>0}\varphi((a-\varepsilon)_+)
\le\varphi(a).
\]
Taking the supremum over these $s$ and using $(FS_c)$
proves the identity.
\end{proof}


We now prove Theorem \ref{Theorem: W*}.
Before the proof, we prepare two lemmas.
The proof of the first lemma can be found in II.4.13 of \cite{MR1873025}, but here we only need a much simpler form.

\begin{lemma}[Non-commutative Egorov theorem]\label{Lemma: Non-commutative Egorov theorem}
Let $M$ be a von Neumann algebra.
If $x_n\in M$ is a sequence such that $x_n\downarrow0$ as $n\to\infty$, then there is a non-zero projection $p\in M$ such that $\|px_np\|\downarrow0$ as $n\to\infty$.
\end{lemma}
\begin{proof}
Let $\omega$ be a normal state on $M$.
By induction on $k$ with $p_0=1\in M$, choose $x_{n_{k+1}}$ and define $p_{k+1}$ such that
\[\omega(p_kx_{n_{k+1}}p_k)<2^{-2(k+1)-1},\qquad p_{k+1}:=1_{[0,2^{-(k+1)}]}(p_kx_{n_{k+1}}p_k)\in p_kMp_k.\]
Let $p:=\inf_kp_k$.
Now 
\[\omega(p_k-p_{k+1})\le2^{k+1}\omega(p_kx_{n_{k+1}}p_k)<2^{-(k+1)-1}.\]
By the normality of $\omega$ we have $\omega(1-p)=\lim_{m\to\infty}\omega(1-p_m)$. So  $p\ne0$ since 
\[\omega(p)=1-\sum_{k=0}^\infty\omega(p_k-p_{k+1})\ge2^{-1}.\]
Then, $px_{n_k}p\le2^{-k}p$ and the monotonicity of $x_n$ give $\|px_np\|\downarrow0$ as $n\to\infty$.
\end{proof}

\begin{lemma}[Non-commutative order-density]\label{Lemma: Non-commutative order-density}
If $A$ is a $\sigma$-weakly dense hereditary C$^*$-subalgebra of $M$, then for any non-zero projection $r\in M$ there is a non-zero projection $p\in A$ such that $p\le r$.
\end{lemma}
\begin{proof}
Choose a non-zero normal state $\omega$ on $M$ such that $s(\omega)\le r$. We may replace $r$ with $r=s(\omega)$.
Since $A$ is $\sigma$-weakly dense in $M$, there is a sequence of positive contractions $h_n\in A$, but not necessarily increasing, such that $\omega(h_n)\to1$ as $n\to\infty$.
Define $h:=\sum_n2^{-n}h_n\in A$, and define an increasing sequence $e_n:=1_{[n^{-1},\infty)}(h)$.  These are in $A$ since  $A$ is hereditary and $e_n\le nh\in A$.
Since $\omega(h_n)\to1$  and $s(h_n) \leq s(h)$, we have $\omega(s(h)) = 1$. Thus $s(\omega) = r\le s(h)$.
This implies $r(1-e_n)r\downarrow0$ $\sigma$-weakly as $n\to\infty$.  By applying the non-commutative Egorov theorem (Lemma \ref{Lemma: Non-commutative Egorov theorem}) on $rMr$, we can find a non-zero projection $p\in rMr$ such that $\|p(1-e_n)p\|\downarrow0.$  Thus $\|(1-e_n)p\|\to 0$, and so $\|e_npe_n-p\|\to0$.
Now $e_npe_n\in A$ since  $A$ is hereditary, and so $p\in A$.  The claim follows.
\end{proof}

\begin{theorem}
Let $\varphi$ be a weight on a von Neumann algebra $M$.
The following are equivalent.
\begin{enumerate}
\item[(i)] $\varphi$ is normal.
\item[(ii)] $\varphi$ is completely additive on projections and norm lower semi-continuous.
\item[(iii)] $\varphi$ is completely additive on projections
and satisfies the spectral finite-spectrum recovery property $(FS_c)$.
\end{enumerate}
\end{theorem}

\begin{proof} Normality implies (ii). The equivalence of (ii) and (iii) follows from Corollary \ref{cor:spectral-recovery-lsc}. Thus it remains for us to prove that (ii) implies (i).

Let $\varphi$ be a weight on a von Neumann algebra $M$ that is norm lower semi-continuous and completely additive on projections.
Let $A:=\overline{\mathfrak{M}}_\varphi^{\|\cdot\|}$ be the norm closure of $\mathfrak{M}_\varphi$ in $M$, a hereditary C$^*$-subalgebra of $M$.
Let $e\in M$ be the support of the $\sigma$-weak closure of $A$ in $M$, a projection such that $eMe=\overline A{}^{\,\sigma\text{-weak}}$.
The restriction $\varphi|_{(eMe)_+}$ is also completely additive on projections and norm lower semi-continuous.
If this is normal, then we can prove $\varphi$ is normal by dividing the cases by whether there is $i$ such that $\varphi(x_i)=\infty$ for a bounded non-degreasing net $x_i\in M_+$.
That is, with $x:=\sup_ix_i$, if $\varphi(x_i)=\infty$ for some $i$ then $\varphi(x)\ge\sup_i\varphi(x_i)=\infty$, otherwise $x_i\in\mathfrak{M}_\varphi\subset eMe$ gives $x\in eMe$ so that $\varphi(x)=\sup_i\varphi(x_i)$.
This allows by replacing $\varphi$ to $\varphi|_{(eMe)_+}$ to assume that $A$ is $\sigma$-weakly dense in $M$, equivalently that $\varphi$ is semi-finite.
Since $\varphi$ is norm lower semi-continuous on $M$, so is its restriction to $A$.
By Combes' theorem, for $F^*:=\{\omega\in A^{*}_+:\omega\le\varphi|_{A_+}\}$, we can write
\[\varphi(a)=\sup_{\omega\in F^*}\omega(a),\qquad a\in A_+.\]

In the inclusion $M_*\subset A^*$ by the semi-finiteness, we show $F^*\subset M_*$ as follows.
Let $\widetilde\omega\in M^*$ be a norm preserving extension of $\omega\in F^*$.
Since $\omega$ is positive, $\widetilde\omega$ is automatically positive by $\|\widetilde\omega\|=\|\omega\|=\sup_i\omega(e_i)$ for an approximate unit $e_i\in A$.
We also have $\widetilde\omega\le\varphi$ by $\mathfrak{M}_\varphi\subset A$.
Since $(\mathfrak{M}_\varphi)_+$ is norm dense in $A_+$ by functional calculus, we may choose a sequence $a_n\in\mathfrak{M}_\varphi$ with $a_n \geq 0$ and $\|a_n\|\le1$ such that $\omega(a_n)\to\|\omega\|$.
Then, $e_n:=1_{[1/2,\infty)}(a_n)\in\mathfrak{M}_\varphi$ satisfies
\[\widetilde\omega(1-e_n)\le2\widetilde\omega(1-a_n)=2(\|\omega\|-\omega(a_n))\to0,\qquad n\to\infty.\]
Since $e_n\in\mathfrak{M}_\varphi$, the restriction $\varphi|_{(e_nMe_n)_+}$ extends to a bounded positive functional on $e_nMe_n$.
It is completely additive on the projections in this
corner, and hence it is normal.
A positive functional dominated by a normal positive functional is normal, and so $\widetilde\omega|_{e_nMe_n}$ is also normal by $\widetilde\omega\le\varphi$.
Composing it with the normal compression
$M\to e_nMe_n:x\mapsto e_nxe_n$ therefore gives a normal positive functional $\omega_n$ on $M$, where
$\omega_n(x):=\widetilde\omega(e_nxe_n)$ for $x\in M$.
By Cauchy-Schwarz on the identity $\widetilde\omega(x-e_nxe_n)=\widetilde\omega((1-e_n)x)+\widetilde\omega(e_nx(1-e_n))$ for $x\in M$, we have 
\[\|\widetilde\omega-\omega_n\|^2\le4\|\widetilde\omega\|\widetilde\omega(1-e_n)\to0,\qquad n\to\infty.\]
Since $M_*$ is norm closed in $M^*$, it follows that 
$\widetilde\omega\in M_*$.
Thus we have $F^*\subset M_*$. 

Let $r\in M$ be any projection.
By the order density lemma (Lemma \ref{Lemma: Non-commutative order-density}) and Zorn's lemma, there is a family $p_k\in A$ of orthogonal projections such that $r=\sum_kp_k$.
Let $p_K\in A$ be the net of all finite partial sums of $\sum_kp_k$ so that $\sum_kp_k=\sup_Kp_K$.
By the complete additivity on orthogonal projections of $\varphi$, the weight $\varphi$ preserves the supremum of $p_K$.
Because $p_K\in A$, we may write $\varphi(p_K)=\sup_{\omega\in F^*}\omega(p_K)$.
Note that each functional $\omega\in F^*$ is normal as we have seen.
Thus,
\[\varphi(r)=\sup_K\varphi(p_K)=\sup_K\sup_{\omega\in F^*}\omega(p_K)=\sup_{\omega\in F^*}\sup_K\omega(p_K)=\sup_{\omega\in F^*}\omega(r).\]

Let $s\in M$ be any positive element of finite spectrum.
If $\varphi(s)<\infty$, then $s\in\mathfrak{M}_\varphi\subset A$ implies the formula $\varphi(s)=\sup_{\omega\in F^*}\omega(s)$ by the Combes theorem.
Assume that $\varphi(s)=\infty$.
Let $\lambda>0$ be the smallest positive value of the spectrum of $s$, and let $r:=1_{(0,\infty)}(s)$.
Then, $s\le\|s\|r$ implies $\varphi(r)=\infty$, and $\lambda r\le s$ implies $\infty=\varphi(r)=\sup_{\omega\in F^*}\omega(r)\le\lambda^{-1}\sup_{\omega\in F^*}\omega(s)$.
Thus, in any cases,
\[\varphi(s)=\sup_{\omega\in F^*}\omega(s).\]

Let $x\in M$ be any positive element.
Let $s_n\in W^*(x)$ be a sequence of spectral step functions such that $s_n\uparrow x$ as $n\to\infty$ in norm.
By the norm lower semi-continuity of $\varphi$, the weight $\varphi$ preserves the limit $s_n\uparrow x$.
Since each $s_n$ has finite spectrum, we have $\varphi(s_n)=\sup_{\omega\in F^*}\omega(s_n)$.
Every positive linear functional preserves norm limits.
Thus,
\[\varphi(x)=\sup_n\varphi(s_n)=\sup_n\sup_{\omega\in F^*}\omega(s_n)=\sup_{\omega\in F^*}\sup_n\omega(s_n)=\sup_{\omega\in F^*}\omega(x).\]
This proves (4) in Theorem \ref{Theorem: W*} for $\varphi$.
Therefore, $\varphi$ is normal.\qedhere

\end{proof}

\begin{corollary}[Haagerup's second problem]
A weight on a von Neumann algebra that is normal on commutative von Neumann subalgebras is normal.
\end{corollary}

\begin{proof} Every commutative C\(^*\)-subalgebra of \(M\) is contained in a unital commutative von Neumann subalgebra. Hence the hypothesis implies that \(\varphi\) preserves norm-convergent positive sums on commutative C\(^*\)-subalgebras. By the earlier theorem, \(\varphi\) is norm lower semi-continuous. Moreover, every mutually orthogonal family of projections lies in a commutative von Neumann subalgebra, so \(\varphi\) is completely additive on projections. the last theorem  therefore implies that \(\varphi\) is normal. \end{proof}

\section{Operator-valued commutative detection}

It is almost immediate that commutative detection for ordinary weights implies commutative detection for Haagerup's operator-valued weights \cite{MR534673}, indeed more generally for extended-positive-valued weights between von Neumann algebras.  We recall that for a von Neumann algebra $N$, its \emph{extended positive cone}
$\widehat{N}_+$ consists of additive, positively homogeneous,
norm-lower-semi-continuous functions
\[
    h \colon (N_*)_+ \longrightarrow [0,\infty],
    \qquad h(0)=0.
\]
Write $\omega(h)=h(\omega)$ for $\omega\in (N_*)_+$.
An element $y\in N_+$ is identified with
$\omega\mapsto\omega(y)$.
Addition, scalar multiplication, and order are pointwise.
This is the functional realization of the extended positive part
used in operator-valued weight theory  (see Section~1 of \cite{MR534673}).

The following then follows immediately from the definitions:

\begin{lemma}
Every increasing net $(h_i)$ in $\widehat{N}_+$ has a supremum
characterized by
\[
    \omega\left(\sup_i h_i\right)
    = \sup_i \omega(h_i), 
    \qquad \omega\in (N_*)_+.
\]
Equality and order in $\widehat{N}_+$ are determined by all the
evaluations at $\omega\in (N_*)_+$.
\end{lemma}

An \emph{extended-positive-valued weight} from $A$ to $N$
is a map $T\colon A_+\to \widehat{N}_+$ 
which vanishes at zero and is additive and positively homogeneous. This of course include the special case of maps into $N_+$. 
For $\omega\in (N_*)_+$, its \emph{scalarization} is the scalar weight
\[
    \varphi_\omega(a)=\omega(T(a)).
\]
For a unital inclusion $N\subseteq M$, an
\emph{operator-valued weight} in Haagerup's sense also satisfies
\[
    T(b^*xb)=b^*T(x)b, 
    \qquad x\in M_+,\ b\in N.
\]
\noindent Here  the element $b^*hb$ is defined by
 $(b^*hb)(\omega)=h(\omega_b)$ for $b \in N$, and 
    $\omega_b(y)=\omega(b^*yb)$.

From the lemma and commutative detection for ordinary weights, the following is immediate: 

\begin{theorem}[Operator-valued commutative determination]
\label{thm:operator-valued-commutative-determination}
Let $T\colon M_+\to\widehat{N}_+$ be additive, positively
homogeneous, and zero at zero. The following conditions are equivalent:
\begin{enumerate}
    \item[(i)]
    $T$ is normal:
    \[
        T(x)=\sup_i T(x_i)
        \qquad\text{whenever } x_i\uparrow x \text{ in } M_+.
    \]

    \item[(ii)]
    $T$ is normal on every unital commutative von Neumann
    subalgebra of $M$.

    \item[(iii)]
    $\omega\circ T$ is a normal scalar weight for every
    $\omega\in (N_*)_+$.

    \item[(iv)]
    $\omega\circ T$ is ultraweakly lower semi-continuous
    for every $\omega\in (N_*)_+$.
\end{enumerate}
In \textup{(ii)}, it suffices to test maximal abelian
von Neumann subalgebras. No faithfulness or semi-finiteness
is assumed.
\end{theorem}

\begin{corollary}
For a unital inclusion $N\subseteq M$, an operator-valued
weight $T\colon M_+\to\widehat{N}_+$ is normal if and only
if its restrictions to commutative von Neumann subalgebras
of $M$ preserve increasing bounded suprema.
\end{corollary}

\bibliographystyle{alpha}
\bibliography{bib}

\end{document}